\documentclass[12pt,reqno]{amsart}
\usepackage{latexsym}
\usepackage{amssymb} 
\usepackage{mathrsfs}
\usepackage{amsmath}
\usepackage{latexsym}
\usepackage{delarray}
\usepackage{amssymb,amsmath,amsfonts,amsthm,
mathrsfs}
\usepackage{graphicx}
\graphicspath{ {./images/} }

\usepackage{hyperref}
\renewcommand\eqref[1]{(\ref{#1})} 

\newtheorem{theorem}{Theorem}[section]
\newtheorem{corollary}[theorem]{Corollary}

\newtheorem{proposition}[theorem]{Proposition}

\theoremstyle{definition}
\newtheorem{remark}[theorem]{Remark}
\newtheorem{example}[theorem]{Example}

\newcommand{\mb}[1]{\ensuremath{\mathbb{#1}}}

\newcommand{\R}{\mb{R}}

\newcommand{\beq}{\begin{equation}}
\newcommand{\eeq}{\end{equation}}

\newcommand{\lara}[1]{\langle #1 \rangle}

\title[Higher order hyperbolic pseudo-differential equations]
{$C^\infty$ well-posedness of higher order hyperbolic pseudo-differential equations with multiplicities}

\author[Claudia Garetto]{Claudia Garetto}
\address{
  Claudia Garetto:
  \endgraf
School of Mathematical Sciences
  \endgraf
  Queen Mary University of London
  \endgraf
 Mile End Road, E1 4UJ, London
  \endgraf
  United Kingdom
  \endgraf
  {\it E-mail address} {\rm c.garetto@qmul.ac.uk}
  }

\author[Bolys Sabitbek]{Bolys Sabitbek}
\address{
  Bolys Sabitbek:
  \endgraf
School of Mathematical Sciences
  \endgraf
 Queen Mary University of London
  \endgraf
 Mile End Road, London, E1 4NS
  \endgraf
  United Kingdom
  \endgraf
  {\it E-mail address} {\rm b.sabitbek@qmul.ac.uk}
  }

\thanks{The authors were supported by the
EPSRC grant EP/V005529/2. The authors declare to have no conflicts of interest.
}
\date{}


\begin{document}

\maketitle

\begin{abstract}
In this paper, we study higher order hyperbolic pseudo-differential equations with variable multiplicities. We work in arbitrary space dimension and we assume that the principal part is time-dependent only. We identify sufficient conditions on the roots and the lower order terms (Levi conditions) under which the corresponding Cauchy problem is $C^\infty$ well-posed. This is achieved via transformation into a first order system, reduction into upper-triangular form and application of suitable Fourier integral operator methods previously developed for hyperbolic non-diagonalisable systems. We also discuss how our result compares with the literature on second and third order hyperbolic equations.
\end{abstract}

\section{Introduction}
In this paper, we study the Cauchy problem for hyperbolic equations of order $m$:
\begin{align}\label{eq-m_order_HE}  
   \begin{cases}
      D_t^mu-\sum_{j=0}^{m-1}A_{m-j}(t,x,D_x)D_t^ju=f(t,x), & t\in[0,T],\, x\in\R^{n},  \\
      D_t^{k-1}u(0,x) = g_k(x), & k=1,\ldots,m,
   \end{cases}  
\end{align}
where $A_{m-j}(t,x,D_x)$ is a pseudo-differential operator of order $m-j$, $D_t = i^{-1}\partial_t$ and $D_x = i^{-1}\partial_x$.  We assume that the principal part of these operators, denoted by $A_{(m-j)}$ is independent of the special variable $x$, i.e. 
	\[
	A_{m-j}(t,x,D_x)=A_{(m-j)}(t,D_x)+(A_{m-j}(t,x,D_x)-A_{(m-j)}(t,D_x)),
	\]
 for all $j=0,\cdots, m-1$. We work under the hypothesis that the roots $\lambda_i(t,\xi)$ of the characteristic polynomial 
\[
\tau^m-\sum_{j=0}^{m-1}A_{(m-j)}(t,\xi)\tau^j = \prod_{i=1}^m (\tau - \lambda_{i}(t,\xi)).
\]
are real-valued functions in  $C^1([0,T], S^1(\R^{2n}))$, hence the equation is hyperbolic with multiplicities or also called weakly hyperbolic.
The right-hand side $f(t,x)$ is assumed to be continuous in $t$ and smooth in $x$.  

The main aim of this paper is to establish under which conditions on the roots and the lower order terms (Levi conditions) the Cauchy problem \eqref{eq-m_order_HE} is $C^\infty$ well-posed. We want our conditions to allow variable multiplicities, i.e., up to order $m$.

In the sequel we give a brief and non exhausting overview of the existing literature on weakly hyperbolic equations.

The well-posedness of the Cauchy problem for weakly hyperbolic equations has been a challenging problem in mathematics since a long time. Unlike strictly hyperbolic equations, weakly hyperbolic equations are characterised by their principal part having real, but not necessarily distinct, roots. This distinction results in more complex and unpredictable behaviour, as illustrated by the following second-order example in one space dimension:
\begin{equation}\label{eq-2HE-example}
\begin{cases}
    \partial_t^2 u - a(t)\partial_x^2 u  + b(t)\partial_x u = 0, \\
    u(0,x)=g_1(x), \quad \partial_t u(0,x) = g_2(x).
\end{cases}
\end{equation}
In \cite{Ol-70} Oleinik investigated the $C^\infty$ well-posedness of the Cauchy problem \eqref{eq-2HE-example} and obtained the sufficient condition
\[
t|b(t)|^2 \leq C a(t) + \partial_t a(t),
\]
which holds uniformly in $t\in[0,T]$ for some positive constant $C$. The understanding of this Cauchy problem was further refined by Colombini and Spagnolo in the celebrated paper \cite{ColSpa-82}, where they constructed a $C^{\infty}$ function $a(t)\geq 0$, such that the homogeneous problem \eqref{eq-2HE-example} (i.e., with $b=0$) is not $C^{\infty}$ well-posed. 

Additional insights into the $C^{\infty}$ well-posedness of second-order hyperbolic equations have been provided by many authors. Nishitani in \cite{Nish-80} explored necessary and sufficient conditions on lower order terms for equations with a single space variable and analytic coefficients. Colombini, De Giorgi, and Spagnolo in \cite{ColDeSpa-79}, and Colombini, Jannelli, and Spagnolo in \cite{ColJanSpa-83}, focused on equations with coefficients dependent solely on time, examining both $C^{\infty}$ and Gevrey well-posedness. In \cite{ColGrOrTa-23} Colombini, Gramchev, Orrù, and Tagilalatela have obtained sufficient conditions for the $C^{\infty}$ well-posedness of the following third order Cauchy problem:
\begin{equation}\label{eq-3HE-example}
    \begin{cases}
        \partial_t^3 u + a_{2,1}(t)\partial_t^2 \partial_x u + a_{1,2}(t)\partial_t \partial_x u+ a_{0,3}(t) \partial_x^3 u + \text{lower order terms} = f(t,x),\\
        u(0,x) = g_1(x), \quad \partial_t  u(0,x) = g_2(x), \quad \partial_t^2  u(0,x) = g_3(x).
    \end{cases}
\end{equation}
Their research particularly focuses on cases where the characteristics roots have constant multiplicity and the coefficients in the principal part are of class $C^2$. In contrast, Wakabayashi in \cite{Wak-15} explored similar problems but with double characteristics and analytic coefficients. For further results on third order hyperbolic equations we refer the reader to \cite{ColOr-99} and \cite{ColTag-06} and the recent extension by Nishitani in \cite{Nish-22}.  
Passing now to arbitrary higher order equations, the majority of the known results of $C^\infty$ well-posedness hold for specific classes of equations: 
\begin{itemize}
\item[-] equations with constant coefficients or with principal part with constant coefficients (see \cite{Gar-51, Hor-85, Sve-70, Wak-80})
\item[-] equations with constant multiplicities or multiplicities up to order two (see \cite{BerPP:12, Cha-74, DeP-72, FlaStr-71, PP:04, PP:09})
\item[-] homogenous equations with coefficients dependent on one variable (either space or time) (see \cite{ColOr-99, ColTag-06, SpaTag-07})
\end{itemize}
To our knowledge general results for non-homogenous higher order hyperbolic equations with $(t,x)$-dependent coefficients and variable multiplicities in any space dimension are still missing.
This paper is the first attempt to provide results of $C^\infty$ well-posedness for a wider class of higher order hyperbolic equations, without restrictive assumptions on the multiplicities or the space dimension.  
The main idea is to transform the higher order equation into a first order system of pseudo-differential equations and then apply our knowledge of non-diagonalisable hyperbolic systems developed in \cite{GarJRuz, GarJRuz2}
to formulate suitable Levi conditions on the lower order terms that will guarantee $C^\infty$ well-posedness. Under this point of view this paper is the natural application of \cite{GarJRuz, GarJRuz2} to higher order hyperbolic equations. 
Note that to better control the lower order terms we work here on equations that are only time-dependent in the principal part. This choice leads to a specific formulation of the Levi conditions on the lower order terms (in line with recent results on $x$-dependnet hyperbolic equations in \cite{SpaTa-22}) however it still allows us to have variable multiplicities of any order, which is a big achievement in comparison to the previous results in this field.

\subsection{Main result}
Before stating our main result let us recall some basic definitions and fix some notations. Let us fix the notations and definitions of symbol classes. 

We say that a function $a(x,\xi)\in C^{\infty}(\mathbb R^n \times \mathbb R^n)$ belongs to $S^m_{1,0}(\mathbb R^n \times \mathbb R^n)$ if there exist constants $C_{\alpha,\beta}$ such that 
\begin{equation*}
    \forall \alpha, \beta \in \mathbb N_0^n \quad : \quad  |\partial_x^{\alpha}\partial_{\xi}a(x,\xi)| \leq C_{\alpha, \beta} \lara{\xi}^{m - \beta} \quad \forall (x,\xi) \in \mathbb R^n \times \mathbb R^n.
\end{equation*}
We denote by $C([0,T],S^m_{0,1}(\mathbb R^n \times \mathbb R^n))$ the space of all symbols $a(t,x,\xi)\in S^m(\mathbb R^n \times \mathbb R^n)$ which are continuous with respect to $t$. If there is no question about the domain under consideration, we will abbreviate the symbol classes by $S_{1,0}^m$ and $C([0,T],S^m_{0,1}(\mathbb R^n \times \mathbb R^n))$, respectively,  or simply by $S^m$ and $CS^m$. 

Going back to the Cauchy problem \eqref{eq-m_order_HE} we will make use of the following notations:
\[
\begin{split}
b_{j}&=A_{m-j+1}(t,x,\xi)\lara{\xi}^{j-m},\\
b_{(j)}&=A_{(m-j+1)}(t,\xi)\lara{\xi}^{j-m},\\
d_{m,k} &= \sum_{j=k}^{m} (b_j - b_{(j)}) \omega_{j,k},
 \end{split}
\]
for $j,k=1,\cdots, m$, where we set $\omega_{k,k}=1$ and 
\begin{align*}
    \omega_{j,k} = \sum_{\alpha_1 + \alpha_2 + \cdots + \alpha_k=j-k} \lambda_1^{\alpha_1}\lambda_2^{\alpha_2}\cdots \lambda_k^{\alpha_k}\lara{\xi}^{k-j}.
\end{align*}
for $j>k$. We then define the transformation matrix 
\begin{align*}
	T(t,\xi)  = \begin{pmatrix}
	1 & 0 & 0 & \ldots & 0 & 0 & 0\\
	\omega_{2,1} & 1 & 0 &  \ldots & 0 & 0 & 0\\
	\omega_{3,1} & \omega_{3,2} & 1 &  \ldots & 0 & 0 & 0 \\
	\vdots & \vdots & \vdots & \ddots & \vdots & \vdots & \vdots \\
	\omega_{m-2,1} & \omega_{m-2,2} & \omega_{m-2,3} & \ldots & 1 & 0 & 0 \\
	\omega_{m-1,1} & \omega_{m-1,2} & \omega_{m-1,3} & \ldots & \omega_{m-1,m-2} & 1 & 0 \\
	\omega_{m,1} & \omega_{m,2} & \omega_{m,3} & \ldots & \omega_{m,m-2}& \omega_{m,m-1} & 1
	\end{pmatrix}.
\end{align*}
which we will prove to be invertible in Proposition \ref{Prop_Schur_decomposition}. The matrix $T^{-1}$ is employed to define the following symbols for $i>j$:
\[
e_{i,j} = \sum_{j < k \leq i} (T^{-1})_{i,k}, D_t\omega_{k,j}.
\]
We are now ready to state our main result and its immediate corollary.
\subsubsection*{\bf Main Theorem}
 \emph{    Let $n\geq 1$ and $m\geq 2$ and consider the Cauchy problem 
     \begin{align*}
   \begin{cases}
      D_t^mu-\sum_{j=0}^{m-1}A_{m-j}(t,x,D_x)D_t^ju=f(t,x), & t\in[0,T],\, x\in\R^{n},  \\
      D_t^{k-1}u(0,x) = g_k(x), & k=1,\ldots,m,
   \end{cases}  
\end{align*}
where $A_{m-j}$ is a pseudo-differential operator of order $m-j$ with principal part $A_{(m-j)}$ independent of $x$ and $f(t,x)$ is in $C([0,T],C^{\infty}(\mathbb R^n))$. Assume that the characteristic polynomial associated to this equation has $m$ real roots $\lambda_i(t,\xi)\in C^1([0,T], S^1(\R^{2n}))$. If the Levi conditions 
\[
\begin{split}
e_{i,j} &\in C([0,T],S^{j-i}(\mathbb{R}^{2n})), \\
d_{m,k} - e_{m,k}&\in C([0,T],S^{k-m}(\mathbb{R}^{2n})),	
\end{split}
\]
    for $i>j$, $2\le i\le m-1$, and for all $k=1,\ldots,m-1$, then the Cauchy problem is $C^\infty$ well-posed.}

 \subsubsection*{\bf Corollary}
\emph{ If the the roots $\lambda_j$ are constant for $j=1,\cdots, m-2$ and 
\[
\begin{split}
d_{m,k}&\in C([0,T],S^{k-m}(\mathbb{R}^{2n})),\\
d_{m,m-1}-D_t\lambda_{m-1}\lara{\xi}&\in C([0,T],S^{-1}(\mathbb{R}^{2n})),
\end{split}
\]
for $k=1,\cdots,m-2$, then the Cauchy problem is $C^\infty$ well-posed.}

\vspace{0.2cm}
 
Note that the formulation of the main theorem is clear in view of the results on non-diagonalisable systems proven in \cite{GarJRuz, GarJRuz2}, however the corollary allow us to formulate easy examples of higher order hyperbolic equations whose Cauchy problem is $C^\infty$ well-posed. Indeed, the Levi conditions above given in terms of symbol order can be translated into conditions on the lower order terms. As an explanatory example, consider the third order equation in space dimension one with characteristic roots 
	\[
	\begin{split}
	\lambda_1 & = \xi,\\
	\lambda_2 & = a(t)\xi,\\
	\lambda_3 & = b(t)\xi.
	\end{split}
	\]
and lower order terms 	
 \[
	a_{0,2}(t,x)D_x^2 + a_{1,1}(t,x)D_xD_t+a_{2,0}(t,x)D_t^2 + a_{0,1}(t,x)D_x + a_{1,0}(t,x)D_t+ a_{0,0}(t,x).
\]
The Levi conditions are fulfilled imposing
 \[
	\begin{split}
	a_{1,1}(t,x) + (1+a(t))a_{2,0}(t,x) - D_ta(t) &= 0,\\
	a_{0,2}(t,x) + a_{1,1}(t,x) + a_{2,0}(t,x)&=0,\\
	a_{0,1}(t,x) + a_{1,0}(t,x)&=0.
	\end{split}
	\]
Note that the specific choice of lower order terms is balanced by variable multiplicities and non-restrictive assumptions on $a(t)$ and $b(t)$, which indeed are simply required to be of class $C^1$. Analogously, one can consider the fourth order example with characteristic roots 

\begin{equation*}
		\lambda_1 = \xi, \,\, \lambda_2 = -\xi, \,\,\, \lambda_3 = \sqrt{a(t)}\xi, \,\,\, \lambda_4 =-\sqrt{a(t)}\xi,
 	\end{equation*}
 	where $a(t)\ge 0$ and lower order terms fulfilling the Levi conditions
  	\begin{align*}
  	&	a_{2,1}(t,x) + \sqrt{a(t)}a_{3,0}(t,x) = D_t\sqrt{a(t)} , \\
  	&   a_{1,2}(t,x) + a_{3,0}(t,x) = 0, \\
  	&  	a_{2,1}(t,x) + a_{0,3}(t,x) = 0, \\
  	&	a_{0,2}(t,x) + a_{2,0}(t,x) = 0, \\
  	&	a_{0,1}(t,x) + a_{1,0}(t,x) = 0,\\
  	& 	a_{1,3} = a_{3,1} = a_{1,1} = 0.
  	\end{align*}
We refer the reader to Example \ref{ex_3_3} and Example \ref{ex_4_4} for more details.

The paper is organised as follows. In Section 2 we collect some preliminaries needed to state and prove our main theorem. In particular, we transform the $m$-order equation into a system of first order pseudo-differential equations and we reduce it into upper-triangular form while a Schur decomposition that will define the matrix $T(x,\xi)$. We also recall the $C^\infty$ result for non-diagonalisable hyperbolic systems proven in \cite{GarJRuz}. Section 3 is dedicated to the proof of our main theorem and corollary and to some explanatory examples. The paper ends with Section 4 where we discuss how our result compares with well-known $C^\infty$ well-posdenss results for second and third order hyperbolic equations and we provide further examples.

\section{Preliminaries}

It is standard procedure when dealing with higher order equations to transform the equation into a system of first order pseudo-differential equations. In detail, making use of 
\[
{u_k=D_t^{k-1}\lara{D_x}^{m-k}u},
\]
for $k=1,\ldots, m$, where $\lara{D_x}$ is the pseudo-differential operator with symbol $\lara{\xi}=(1+|\xi|^2)^{\frac{1}{2}}$, the Cauchy problem \eqref{eq-m_order_HE} can be re-written as
\begin{equation}
\label{system_HE}
\begin{split}
D_tU&=A(t, D_x)U+B(t,x,D_x)U+F,\\
U(0,x)&=(\lara{D_x}^{m-1}g_1,\lara{D_x}^{m-2}g_2,\cdots g_m)^T
\end{split}
\end{equation}
where $U=(u_1,u_2,\cdots,u_m)^T$, 
\begin{align}\label{matrix-A}
A(t,\xi) =& \begin{pmatrix}
0 & \lara{\xi} & 0& \dots & 0 &0\\
0 & 0 & \lara{\xi} & \dots & 0 & 0\\
0 & 0 & 0 & \ddots & 0 & 0 \\
\vdots & \vdots & \vdots & \dots & \vdots & \vdots \\
0 & 0 & 0 &\dots & 0 &\lara{\xi} \\
b_{(1)} & b_{(2)} & b_{(3)}& \dots &b_{(m-1)} & b_{(m)}
\end{pmatrix},
\end{align}
\begin{align}\label{matrix-B}
B(t,x,\xi) =& \begin{pmatrix}
0 & 0 & 0& \dots & 0 &0\\
0 & 0 & 0 & \dots & 0 & 0\\
0 & 0 & 0 & \ddots & 0 & 0 \\
\vdots & \vdots & \vdots & \dots & \vdots & \vdots \\
0 & 0 & 0 &\dots & 0 & 0 \\
b_1-b_{(1)} & b_2-b_{(2)} & b_3-b_{(3)}& \dots &b_{m-1}-b_{(m-1)} & b_m-b_{(m)}
\end{pmatrix},
\end{align}
with
\[
\begin{split}
b_{j}&=A_{m-j+1}(t,x,\xi)\lara{\xi}^{j-m},\\
b_{(j)}&=A_{(m-j+1)}(t,\xi)\lara{\xi}^{j-m},
  \end{split}
\]
for $j=1,\cdots, m$, and $F=(0,0,\cdots, 0, f)^T$. The matrix $A$ is in Sylvester form and has the roots $\lambda_j(t,\xi)$'s as eigenvalues. In the rest of the paper we will therefore focus on the Cauchy problem \eqref{system_HE}. As a first step to prove $C^\infty$ well-posedness we will reduce the matrix $A$ into upper-triangular form. This is a known as Schur decomposition. In our particular case the transformation matrix $T$ is defined by symbols of order $0$ as proven in the following proposition.

\subsection{Schur decomposition}
 
\begin{proposition}\label{Prop_Schur_decomposition}
    Let $A(t,\xi)$ be a $m\times m$ matrix valued symbol $A(t,\xi)$ as defined in \eqref{matrix-A} with real eigenvalues $\lambda_1,\ldots,\lambda_m$.
    There exists a unitary $m \times m$ matrix valued symbol $T(t,\xi)$ such that $J = T^{-1}AT$ is an upper triangular matrix with diagonal elements $\lambda_1,\ldots,\lambda_m$.
    More precisely,  
   \begin{align*}
	T^{-1} A T = \begin{pmatrix}
	\lambda_1 & \lara{\xi} & 0 & \cdots & 0 & 0 \\
	0 & \lambda_2 & \lara{\xi} &  \cdots & 0 & 0 \\
	0 & 0 & \lambda_3 & \cdots & 0 & 0 \\
	\vdots & \vdots & \vdots & \ddots & \vdots & \vdots \\
	0 & 0 & 0 & \cdots & \lambda_{m-1} & \lara{\xi} \\
	0 & 0 & 0 & \cdots & 0 & \lambda_{m}
	\end{pmatrix},
\end{align*}
with
\begin{align*}
	T(t,\xi)  = \begin{pmatrix}
	1 & 0 & 0 & \ldots & 0 & 0 & 0\\
	\omega_{2,1} & 1 & 0 &  \ldots & 0 & 0 & 0\\
	\omega_{3,1} & \omega_{3,2} & 1 &  \ldots & 0 & 0 & 0 \\
	\vdots & \vdots & \vdots & \ddots & \vdots & \vdots & \vdots \\
	\omega_{m-2,1} & \omega_{m-2,2} & \omega_{m-2,3} & \ldots & 1 & 0 & 0 \\
	\omega_{m-1,1} & \omega_{m-1,2} & \omega_{m-1,3} & \ldots & \omega_{m-1,m-2} & 1 & 0 \\
	\omega_{m,1} & \omega_{m,2} & \omega_{m,3} & \ldots & \omega_{m,m-2}& \omega_{m,m-1} & 1
	\end{pmatrix}.
\end{align*}
$T(x,\xi)$ has determinant $1$ and the inverse matrix $T^{-1}$ is lower triangular with entries $1$ on the diagonal and 
\begin{align*}
	(T^{-1})_{i,j} = - \omega_{i,j} - \sum_{k=j+1}^{i-1} \omega_{i,k}(T^{-1})_{k,j},
\end{align*}
for $i>j$, where 
\begin{align*}
    \omega_{j,k} = \sum_{\alpha_1 + \alpha_2 + \cdots + \alpha_k=j-k}  \lambda_1^{\alpha_1}\lambda_2^{\alpha_2}\cdots \lambda_k^{\alpha_k}\lara{\xi}^{k-j},
\end{align*}
for $j>k$, with $\alpha_1,\cdots\alpha_k\in\mathbb{N}_0$.
\end{proposition}
\begin{proof}[Proof of Proposition \ref{Prop_Schur_decomposition}]
This is a constructive proof. We aim to identify a matrix $T$ satisfying the equation $AT=TJ$, where $A(t,\xi)$ is defined in \eqref{matrix-A} and
    \begin{align*}
        J(t,\xi) = \begin{pmatrix}
            \lambda_1(t,\xi) & \lara{\xi} & 0 & \dots & 0 & 0 \\
            0 & \lambda_2(t,\xi) & \lara{\xi} &  \dots & 0 & 0 \\
            0 & 0 & \lambda_3 (t,\xi) &  \dots & 0 & 0 \\
            \vdots & \vdots & \vdots & \ddots & \vdots & \vdots \\
            0 & 0 & 0 & \dots & \lambda_{m-1}(t,\xi)  & \lara{\xi} \\
            0 & 0 & 0 & \dots & 0 &\lambda_{m} (t,\xi) 
        \end{pmatrix}.
    \end{align*}
    The relation $AT=TJ$ can be expressed as follows:
    \begin{align*}
        &A \begin{pmatrix}
            T_1 & T_2 & T_3 & \dots & T_{m-1} & T_m
        \end{pmatrix} = \begin{pmatrix}
            T_1 & T_2 & T_3 & \dots & T_{m-1} & T_m
        \end{pmatrix} J \\
        &=\begin{pmatrix}
            \lambda_1 T_1 & T_1\lara{\xi} + \lambda_2 T_2 & T_2\lara{\xi} + \lambda_3 T_3 & \dots & T_{m-2}\lara{\xi} + \lambda_{m-1} T_{m-1} &T_{m-1}\lara{\xi} + \lambda_m T_m
        \end{pmatrix},
    \end{align*}
    where $T_j$, for $j=1,\ldots,m$, denotes the $j$-th column vector of the matrix $T$. We denote the entries of $T_j$ with $\omega_{h,j}$, with $h=1,\cdots,m$. To determine the matrix $T$, we systematically solve the following sequence of equations:
    \begin{align*}
        (A - \lambda_1 I)T_1 &= 0,\\
        (A - \lambda_2 I)T_2 &= T_1\lara{\xi},\\
        (A - \lambda_3 I)T_3 &= T_2\lara{\xi},\\
        \vdots \\
        (A - \lambda_{m-1} I)T_{m-1} &= T_{m-2}\lara{\xi},\\
        (A - \lambda_m I)T_{m} &= T_{m-1}\lara{\xi}.
    \end{align*}
    This will lead us to formulate $\omega_{h,j}$ in terms of the eigenvalues of $A$ and powers of $\lara{\xi}$.
    
 \textbf{Step 1.} Solving the equation $(A - \lambda_1 I)T_1 =0$ yields: 
 \begin{align*}
     -\lambda_1 \omega_{1,1} + \lara{\xi} \omega_{2,1} = 0, \quad &\Longrightarrow  \quad \omega_{1,1} =1 \quad \text{ and } \quad \omega_{2,1} = \lambda_1\lara{\xi}^{-1}, \\
     -\lambda_1 \omega_{2,1} + \lara{\xi} \omega_{3,1} = 0, \quad &\Longrightarrow  \quad  \omega_{3,1} = \lambda_1^2\lara{\xi}^{-2}, \\
      -\lambda_1 \omega_{3,1} + \lara{\xi} \omega_{4,1} = 0, \quad &\Longrightarrow  \quad \omega_{4,1} = \lambda_1^3\lara{\xi}^{-3}, \\
      & \vdots \\
      -\lambda_1 \omega_{m-1,1} + \lara{\xi} \omega_{m,1} = 0, \quad &\Longrightarrow  \quad \omega_{m,1} = \lambda_1^{m-1}\lara{\xi}^{-m+1},
 \end{align*}
 and the validity of the last term 
 \begin{equation*}
     b_{(1)}\omega_{1,1} +  b_{(2)}\omega_{2,1} +  b_{(3)}\omega_{3,1} + \cdots  + b_{(m-1)}\omega_{m-1,1} +  \omega_{m,1}(b_{(m)} - \lambda_1) = 0
 \end{equation*}
 is ensured by $\det (A-\lambda_1I)=0$.
Hence, we can write the first column vector $T_1$ as
 \begin{align*}
     T_1 = \begin{pmatrix}
         1 & \omega_{2,1} & \omega_{3,1} & \ldots & \omega_{m-1,1} &   \omega_{m,1}
     \end{pmatrix}^T,
 \end{align*}
 where the coefficients $\omega_{j,1}$ are given by 
 \begin{align*}
    \omega_{j,1} =  \lambda_1^{j-1}\lara{\xi}^{1-j},
\end{align*}
for $j=2,\ldots,m$.

\textbf{Step 2.} In a similar manner, solving the equation $(A - \lambda_2 I)T_2 = T_1\lara{\xi}$ yields:  
\begin{align*}
     -\lambda_2 \omega_{1,2} + \lara{\xi} \omega_{2,2} = \lara{\xi}, \quad &\Longrightarrow  \quad \omega_{1,2} =0 \quad \text{ and } \quad \omega_{2,2} = 1 \\
     -\lambda_2 \omega_{2,2} + \lara{\xi} \omega_{3,2} = \lambda_1 \quad &\Longrightarrow  \quad  \omega_{3,2} = (\lambda_1 + \lambda_2)\lara{\xi}^{-1}, \\
      -\lambda_2 \omega_{3,2} + \lara{\xi} \omega_{4,2} = \lambda_1^2\lara{\xi}^{-1} \quad &\Longrightarrow  \quad \omega_{4,2} = (\lambda_1^2 + \lambda_1\lambda_2 + \lambda_2^2)\lara{\xi}^{-2}, \\
      & \vdots \\
      -\lambda_2 \omega_{m-1,2} + \lara{\xi} \omega_{m,2} = \lambda_1^{m-1}\lara{\xi}^{2-m} \quad &\Longrightarrow  \quad \omega_{m,2} =  \sum_{\alpha_1 + \alpha_2 =m-2} \lambda_1^{\alpha_1}\lambda_2^{\alpha_2}\lara{\xi}^{2-m}, 
 \end{align*}
 and the validity of the last term
 \begin{equation*}
     b_{(1)}\omega_{1,2} +  b_{(2)}\omega_{2,2} +  b_{(3)}\omega_{3,2} + \cdots  + b_{(m-1)}\omega_{m-1,2} +  \omega_{m,2}(b_{(m)} - \lambda_2) = \lambda_1^{m-1}\lara{\xi}^{2-m},
 \end{equation*}
 is confirmed by employing the characteristic equation  
 $$\tau^m - \sum_{j=1}^m b_{(j)}\lara{\xi}^{m-j}\tau^{j-1} = \prod_{j=1}^m(\tau - \lambda_j),$$ 
and writing each $b_{(j)}$ for $j=1,\ldots,m$ in terms of the characteristic roots $\lambda_1,\ldots,\lambda_m $.
Hence, the second column vector $T_2$ is can be written as
\begin{align*}
     T_2 = \begin{pmatrix}
         0 & 1 & \omega_{3,2} & \omega_{4,2} & \ldots & \omega_{m-1,2} &   \omega_{m,2}
     \end{pmatrix}^T,
 \end{align*}
 where  
 \begin{align*}
    \omega_{j,2} = \sum_{\alpha_1 + \alpha_2 =j-2} \lambda_1^{\alpha_1}\lambda_2^{\alpha_2}\lara{\xi}^{2-j},
\end{align*}
for $j=3,\ldots,m$.

\textbf{Step 3.} Continuing with this process, by solving the equation $(A - \lambda_3 I)T_3 = T_2\lara{\xi}$, we determine the third column vector $T_3$ as   
\begin{align*}
     T_3 = \begin{pmatrix}
         0 & 0 & 1 & \omega_{4,3} & \omega_{5,3} & \ldots & \omega_{m-1,3} &   \omega_{m,3}
     \end{pmatrix}^T,
 \end{align*}
 where for $j=4,\ldots,m$, the coefficients $\omega_{j,3}$ are given by 
 \begin{align*}
    \omega_{j,3} = \sum_{\alpha_1 + \alpha_2 + \alpha_3 =j-3} \lambda_1^{\alpha_1}\lambda_2^{\alpha_2}\lambda_3^{\alpha_3}\lara{\xi}^{3-j}.
\end{align*}
 Analogously, from the equation $(A - \lambda_j I)T_j = T_{j-1}\lara{\xi}$ we get the entries of $T_j$ with $\omega_{j,j}=1$ and 
 \begin{align*}
    \omega_{j,k} = \sum_{\alpha_1 + \alpha_2 + \cdots + \alpha_k=j-k}  \lambda_1^{\alpha_1}\lambda_2^{\alpha_2}\cdots \lambda_k^{\alpha_k}\lara{\xi}^{k-j},
\end{align*}
for $j>k$, with $\alpha_1,\cdots\alpha_k\in\mathbb{N}_0$.

\textbf{Step m-1:} In the penultimate step, solving the equation $(A - \lambda_{m-1} I)T_{m-1} = T_{m-2}\lara{\xi}$ yields the column vector $T_{m-1}$ as  
\begin{align*}
     T_{m-1} = \begin{pmatrix}
         0 & 0 & 0 & \ldots & 1 & \omega_{m,m-1}
     \end{pmatrix}^T,
 \end{align*}
 where the coefficient $\omega_{m,m-1}$ is determined by
 \begin{align*}
    \omega_{m,m-1} &= \sum_{\alpha_1 + \alpha_2 + \dots +\alpha_{m-1} = 1}  \lambda_1^{\alpha_1}\lambda_2^{\alpha_2}\lambda_3^{\alpha_3}\cdots\lambda_{m-1}^{\alpha_{m-1}}\lara{\xi}^{-1} \\
    & = (\lambda_1 + \lambda_2 + \lambda_3 + \cdots + \lambda_{m-1})\lara{\xi}^{-1}.
\end{align*}
\textbf{Step m.} In the final step, solving $(A - \lambda_{m} I)T_{m} = T_{m-1}\lara{\xi}$ allows us to determine the column vector $T_{m}$ as   
\begin{align*}
     T_{m} = \begin{pmatrix}
         0 & 0 & 0 & \ldots & 0 & 1 
     \end{pmatrix}^T.
 \end{align*}
Given that $T$ is a unitary matrix with determinant $1$, it is invertible. The matrix $T^{-1}$ is lower-triangular, identically $1$ on the diagonal and  with entries  \begin{align*}
	(T^{-1})_{i,j} = - \omega_{i,j} - \sum_{k=j+1}^{i-1} \omega_{i,k}(T^{-1})_{k,j},
\end{align*}
for $i>j$. This concludes the proof.
\end{proof}
\begin{remark}
Note that the entries of the matrices $T$ and $T^{-1}$ are polynomials in the roots $\lambda_i$, $i=1,\cdots, n$, they are also symbols of order $0$, independent of $x$ and of class $C^1$ with respect to $t$.
\end{remark}
\begin{remark}
\label{remark_ut}
Proposition \ref{Prop_Schur_decomposition} holds also when the matrix $A$ depends on $x$. However, in this case the equality $J=T^{-1}AT$, valid at the level of the symbol matrices, can be transferred to the level of the operators by adding a matrix of remainders terms of order $0$. This is a direct consequence of the symbolic calculus. 
\end{remark}

\begin{example}
As explanatory examples, we write down explicitly $T$ and $T^{-1}$ when $m=2,3,4$.  

\begin{itemize}
    \item For the case $m=2$, the matrix 
$T$ and its inverse $T^{-1}$ are given by
    \begin{align*}
        T =  \begin{pmatrix}
            1 & 0 \\
            \lambda_1 \lara{\xi}^{-1} & 1
        \end{pmatrix} \quad \text{ and } \quad T^{-1} = \begin{pmatrix}
            1 & 0 \\
            - \lambda_1 \lara{\xi}^{-1} & 1
        \end{pmatrix}.
    \end{align*}
    \item For the case $m=3$, we have  
    \begin{align*}
        T =  \begin{pmatrix}
            1 & 0 & 0\\
           \lambda_1 \lara{\xi}^{-1} & 1 & 0\\
            \lambda_1^2 \lara{\xi}^{-2} & (\lambda_1 + \lambda_2)\lara{\xi}^{-1} & 1
        \end{pmatrix}
    \end{align*}
    and 
    \begin{align*}
 T^{-1} = \begin{pmatrix}
            1 & 0 & 0\\
     -\lambda_1\lara{\xi}^{-1} & 1 & 0\\
  \lambda_1\lambda_2\lara{\xi}^{-2}& -(\lambda_1+\lambda_2)\lara{\xi}^{-1} & 1
        \end{pmatrix}.
    \end{align*}
    \item  For the case $m=4$,  
    \begin{align*}
    T = \begin{pmatrix}
	1 & 0 & 0 & 0\\
\lambda_1\lara{\xi}^{-1} & 1 & 0 & 0\\
\lambda_1^2\lara{\xi}^{-2}&  (\lambda_1+\lambda_2)\lara{\xi}^{-1} & 1 &0\\
\lambda_1^3\lara{\xi}^{-3}& (\lambda_1^2+\lambda_2^2 + \lambda_1\lambda_2)\lara{\xi}^{-2} & (\lambda_1 + \lambda_2 +\lambda_3)\lara{\xi}^{-1}  & 1
	\end{pmatrix},
\end{align*}
and 
\begin{align*}
T^{-1} = \begin{pmatrix}
	1 & 0 & 0 & 0 \\
	-\lambda_1 \lara{\xi}^{-1} & 1 & 0 & 0 \\
	\lambda_1\lambda_2 \lara{\xi}^{-2} & -(\lambda_2+ \lambda_1)\lara{\xi}^{-1} & 1& 0 \\
	-\lambda_1\lambda_2\lambda_3\lara{\xi}^{-3} & (\lambda_1\lambda_2 + \lambda_1\lambda_3 + \lambda_2\lambda_3)\lara{\xi}^{-2} & -(\lambda_1 + \lambda_2 + \lambda_3)\lara{\xi}^{-1} & 1
	\end{pmatrix}.
\end{align*}
\end{itemize}
\end{example}
We conclude this section by recalling the $C^\infty$ well-posedness result proven for hyperbolic systems which in general are not diagonalisable but have principal part in upper triangular form in \cite{GarJRuz}. This result will be applied to our system after transformation via the matrix $T$.
\subsection{$C^\infty$ well-posdenss of non-diagonalisable hyperbolic systems }
\begin{theorem}[Garetto-Jäh-Ruzhansky in \cite{GarJRuz}]\label{thm-GJR}
    Let $n\geq 1$, $m\geq 2$, and let
    \begin{align}\label{eq-thm_GJR}
        \begin{cases}
 D_t V = P(t,x,D_x) V + L(t,x,D_x)V + F(t,x),&  (t,x) \in [0,T]\times \mathbb{R}^n, \\
 V(0,x) = V_0(x),& x \in \mathbb{R}^n,
\end{cases}
    \end{align}
    where $P(t,x,D_x) \in (CS^1)^{m\times m} $ is an upper-triangular matrix of pseudo-differential operators of order $1$ and $L(t,x,D_x)\in (CS^0)^{m\times m}$ is a matrix of pseudo-differential operators of order $0$, continuous with respect to $t$. Hence, if 
    \begin{itemize}
        \item[] the lower order terms $\ell_{i,j}$ belong to $C([0,T], \Psi^{j-i})$ for $i>j$,
    \end{itemize}
    $V_k^0 \in H^{s+k-1}(\mathbb R^n)$ and $F_k \in C([0,T], H^{s+k-1})$ for $k=1,\ldots,m$, then \eqref{eq-thm_GJR} has a unique anisotropic Sobolev solution $V$, i.e. $V_k \in C([0,T], H^{s+k-1})$ for $k=1,\ldots,m$. 
\end{theorem}
From Theorem \ref{thm-GJR} it follows immediately that the Cauchy problem is $C^\infty$ well-posed when we choose right-hand side and initial data that are smooth with respect to $x$ and compactly supported. 

\section{$C^\infty$ well-posedness of the Cauchy problem for higher order hyperbolic equations with time-dependent principal part}
This section is devoted to the proof of our main result: the $C^\infty$ well-posedness of the hyperbolic Cauchy problem 
\begin{equation}\label{eq-mHE} 
\begin{cases}
    D_t^mu - \sum_{j=0}^{m-1} A_{m-j}(t,x,D_x)D_t^j u = f(t,x), & t\in [0,T], x \in \mathbb R^n, \\
    D_t^{k-1} u(0,x) = g_k(x), & k=1,2,\ldots,m.
\end{cases}
\end{equation}
where the principal part $A_{(m-j)}$ of the operators $A_{m-j}$ is $x$-independent. We assume that the characteristic polynomial
 \begin{equation}\label{eq-HCP} 
     \tau^m-\sum_{j=0}^{m-1} A_{(m-j)}(t,\xi)\tau^j=\prod_{j=1}^m (\tau-\lambda_j(t,\xi)),
 \end{equation}
has real-valued roots $\lambda_1, \ldots, \lambda_m \in C^1([0,T],S^1(\mathbb R^{2n}))$.  By Section 2, we will equivalently study the first order system  
\begin{equation}\label{eq-mHS} 
    \begin{cases}
        D_t U = A(t,D_x) U + B(t,x,D_x)U + F(t,x), & (t,x)\in [0,T]\times \mathbb R^n,\\
        U|_{t=0} = U_0, & x \in \mathbb R^n,
    \end{cases}
\end{equation}
where $U_0 = (u_1(0), u_2(0), \ldots, u_m(0))^T$, $F(t,x)= (0, 0, \ldots, f(t,x))^T$, with the operators $A(t,D_x)$ and $B(t,x,D_x)$ given by 
\begin{align*}
A(t,D_x) =& \begin{pmatrix}
0 & \lara{D_x} & 0& \dots & 0 &0\\
0 & 0 & \lara{D_x}& \dots & 0 & 0\\
0 & 0 & 0 & \ddots & 0 & 0 \\
\vdots & \vdots & \vdots & \dots & \vdots & \vdots \\
0 & 0 & 0 &\dots & 0 &\lara{D_x} \\
b_{(1)} & b_{(2)} & b_{(3)}& \dots &b_{(m-1)} & b_{(m)}
\end{pmatrix}
\end{align*}
and
\begin{align*}
    B(t,x,D_x) = \begin{pmatrix}
0 & 0 &  \dots & 0 &0\\
0 & 0 &  \dots & 0 & 0\\
0 & 0 &  \ddots & 0 & 0 \\
\vdots & \vdots & \vdots & \vdots & \vdots  \\
0 & 0 & \dots & 0 & 0  \\
b_1 - b_{(1)} & b_2 - b_{(2)} &  \dots & b_{m-1} - b_{(m-1)} & b_m - b_{(m)}
\end{pmatrix}.
\end{align*}
The entries of $A(t,D_x)$ and $B(t,x,D_x)$ are given by  
\begin{align*}
    b_j = A_{m-j+1}(t,D_x)\lara{D_x}^{j-m} \quad & \text{ and } \quad  b_{(j)}=A_{(m-j+1)}(t,D_x)\lara{D_x}^{j-m}
\end{align*}
for $j=1,2,\dots,m$

\subsection{Transformation to upper-triangular form} We now employ Proposition \ref{Prop_Schur_decomposition} to transform the system \eqref{eq-mHE} into upper-triangular form, i.e., we set $U=TV$.  This is a significant step as it allows us to apply Theorem \ref{thm-GJR} and to find the conditions which guarantee anisotropic Sobolev well-posedness. In detail, we can write  
\begin{align*}
D_tV=T^{-1}ATV+T^{-1}BTV-T^{-1}D_tT V + T^{-1}F,
\end{align*}
and 
$V|_{t=0}=T^{-1}|_{t=0}U_0=V_0$. Passing now at the symbol level, we have that the principal operator $(T^{-1}AT)(t,D_x)$ has symbol
\begin{align*}
	P = T^{-1} A T = \begin{pmatrix}
	\lambda_1 & \lara{\xi} & 0 & \cdots & 0 & 0 \\
	0 & \lambda_2 & \lara{\xi} &  \cdots & 0 & 0 \\
	0 & 0 & \lambda_3 & \cdots & 0 & 0 \\
	\vdots & \vdots & \vdots & \ddots & \vdots & \vdots \\
	0 & 0 & 0 & \cdots & \lambda_{m-1} & \lara{\xi} \\
	0 & 0 & 0 & \cdots & 0 & \lambda_{m}
	\end{pmatrix},
\end{align*}
in upper-triangular form. The zero-order operator $T^{-1}BT-T^{-1}D_tT$ can be written as $BT-T^{-1}D_tT$ since $T^{-1}B=B$. It follows that we do not create any remainder terms in this composition of these operators and that $T^{-1}BT-T^{-1}D_tT$ has symbol matrix
\[
D(t,x,D_x)-E(t,D_x)
\]
where  
\begin{align*}
	D= T^{-1}BT = \begin{pmatrix}
	0 & 0 & 0 & \cdots & 0 & 0 \\
	0 & 0 & 0 & \cdots & 0 & 0 \\
	\vdots & \vdots & \vdots & \ddots & \vdots & \vdots \\
	0 & 0 & 0 & \cdots & 0 & 0 \\
	d_{m,1} & d_{m,2} & d_{m,3} & \cdots & d_{m, m-1} & d_{m,m} \\
	\end{pmatrix},	
\end{align*}
\begin{align*}
E =	T^{-1}D_tT = \begin{pmatrix}
	0 & 0 & 0 & \ldots & 0 & 0 \\
 	e_{2,1} & 0 & 0 & \ldots & 0 & 0 \\
	e_{3,1} & e_{3,2} & 0 & \ldots & 0 & 0 \\
	\vdots & \vdots & \vdots & \ddots & \vdots & \vdots \\
	e_{m-1,1} & 	e_{m-1,2} & 	e_{m-1,3} & \ldots & 0 & 0 \\
	e_{m,1} & 	e_{m,2} & 	e_{m,3} & \ldots & 	e_{m,m-1} & 0
	\end{pmatrix}.
\end{align*}
\begin{align}\label{eq-tranformed_systemm}
    D_t V = P(t,D_x) V + L(t,x,D_x) V + T^{-1}F,
\end{align}
where $P(t,D_x) = T^{-1}AT$ and $L(t,x,D_x) =( D-E)(t,x,D_x)$. By direct computations we obtain the following auxiliary result.
\begin{proposition}
The entries of the matrices $D$ and $E$ are given by 
\begin{align*}
	d_{m,k} = \sum_{j=k}^{m} (b_j - b_{(j)}) \omega_{j,k} 
\end{align*}
for $k=1,\ldots,m$ and 
\begin{align*}
	e_{i,j} = \sum_{j < k \leq i} (T^{-1})_{i,k}, D_t\omega_{k,j}, 
\end{align*}
for $i>j$, where
\begin{align*}
    \omega_{k,j} = \sum_{\alpha_1 + \alpha_2 + \cdots + \alpha_j=k-j} \lambda_1^{\alpha_1}\lambda_2^{\alpha_2}\cdots \lambda_j^{\alpha_j}\lara{\xi}^{j-k}.
\end{align*}
\end{proposition}
\begin{proof}
By definition of the matrices $T$, $T^{-1}$ and $B$, we immediately see that $D=T^{-1}BT=BT$ and the only non-zero row is the $m$-th row. Hence, 
\begin{align*}
	d_{m,k} = \sum_{j=k}^{m} (b_j - b_{(j)}) \omega_{j,k} =b_k-b_{(k)}+\sum_{j=k+1}^{m} (b_j - b_{(j)}) \omega_{j,k},
\end{align*}
for $k=1,\ldots,m$. By definition of $T^{-1}$ and $T$ we have that $E$ is a strictly lower triangular matrix. From Proposition \ref{Prop_Schur_decomposition}, we have that 
\[
e_{i,j} = \sum_{j < k \leq i} (T^{-1})_{i,k}, D_t\omega_{k,j}, 
\]
for $i>j$, where
\begin{align*}
   \omega_{k,j} = \sum_{\alpha_1 + \alpha_2 + \cdots + \alpha_j=k-j} \lambda_1^{\alpha_1}\lambda_2^{\alpha_2}\cdots \lambda_j^{\alpha_j}\lara{\xi}^{j-k}.
\end{align*}
\end{proof}
We now apply Theorem \ref{thm-GJR}, to the system \eqref{eq-tranformed_systemm} with initial condition $V_0$. We immediately obtain anisotropic Sobolev well-posedness of any order and therefore $C^\infty$ well-posedness if the $(i,j)$-entry of the matrix $L$ is given by pseudodifferential operators of order $j-i$ below the diagonal, i.e., when $i>j$. Since $L(t,x,D_x) =(D-E)(t,x,D_x)$ this means to request
\begin{equation}
\label{LC-1}
\begin{split}
e_{i,j} &\in C([0,T],S^{j-i}(\mathbb{R}^{2n})), \\
d_{m,k} - e_{m,k}&\in C([0,T],S^{k-m}(\mathbb{R}^{2n})),	
\end{split}
\end{equation}
for $i>j$, $2\le i\le m-1$, and for all $k=1,\ldots,m-1$. 

We have therefore proven the following theorem.
\begin{theorem}
 \label{main_theo}
     Let $n\geq 1$ and $m\geq 2$ and consider the Cauchy problem 
     \begin{align*}
   \begin{cases}
      D_t^mu-\sum_{j=0}^{m-1}A_{m-j}(t,x,D_x)D_t^ju=f(t,x), & t\in[0,T],\, x\in\R^{n},  \\
      D_t^{k-1}u(0,x) = g_k(x), & k=1,\ldots,m,
   \end{cases}  
\end{align*}
where $A_{m-j}$ is a pseudo-differential operator of order $m-j$ with principal part $A_{(m-j)}$ independent of $x$ and $f(t,x)$ is in $C([0,T],C^{\infty}(\mathbb R^n))$. Assume that the characteristic polynomial associated to this equation has $m$ real roots $\lambda_i(t,\xi)\in C^1([0,T], S^1(\R^{2n}))$. If the Levi conditions 
\[
\begin{split}
e_{i,j} &\in C([0,T],S^{j-i}(\mathbb{R}^{2n})), \\
d_{m,k} - e_{m,k}&\in C([0,T],S^{k-m}(\mathbb{R}^{2n})),	
\end{split}
\]
for $i>j$, $2\le i\le m-1$, and for all $k=1,\ldots,m-1$, then the Cauchy problem is $C^\infty$ well-posed.
\end{theorem}
A careful analysis of the Levi condition on $e_{i,j}$ leads us to other two immediate corollaries of Theorem \ref{main_theo}. It suffices to notice that the terms $D_j\lambda_h$, $h=1,\cdots, j$, appear in the definition of $e_{i,j}$ since
\[
e_{i,j} = \sum_{j < k \leq i} (T^{-1})_{i,k}, D_t\omega_{k,j}=\sum_{j < k \leq i} (T^{-1})_{i,k} D_t\biggl(\sum_{\alpha_1 + \alpha_2 + \cdots + \alpha_j=k-j} \lambda_1^{\alpha_1}\lambda_2^{\alpha_2}\cdots \lambda_j^{\alpha_j}\lara{\xi}^{j-k}\biggr),
\]
for $i>j$. Hence, if the roots $\lambda_1, \lambda_2,\cdots,\lambda_j$ are constant (in time) than the entries $e_{i,j}$ vanish and trivially fulfil our Levi conditions. In addition,
\[
e_{m,m-1}=(T^{-1})_{m,m}D_t\omega_{m,m-1}=D_t\omega_{m,m-1}=\sum_{h=1}^{m-1}D_t\lambda_h\lara{\xi}=D_t\lambda_{m-1}\lara{\xi}.
\]


\begin{corollary}
\label{cor_roots}
If the the roots $\lambda_j$ are constant for $j=1,\cdots, m-2$ and 
\[
\begin{split}
d_{m,k}&\in C([0,T],S^{k-m}(\mathbb{R}^{2n})),\\
d_{m,m-1}-D_t\lambda_{m-1}\lara{\xi}&\in C([0,T],S^{-1}(\mathbb{R}^{2n})),
\end{split}
\]
for $k=1,\cdots,m-2$, then the Cauchy problem is $C^\infty$ well-posed.
\end{corollary}
To illustrate our result we discuss in detail the $m=3$ toy model.

\subsection{Third-Order Hyperbolic Equations with $(t,x)$-dependent lower order terms}

\leavevmode

Let us consider the Cauchy problem
\begin{equation}\label{eq-3HE} 
\begin{cases}
    D_t^3u - \sum_{j=0}^{2} A_{3-j}(t,x,D_x)D_t^j u = f(t,x), & t\in [0,T], x \in \mathbb R^n, \\
    D_t^{k-1} u(0,x) = g_k(x), & k=1,2,3.
\end{cases}
\end{equation}
where the principal part $A_{(3-j)}$ of the operators $A_{3-j}$ is $x$-independent. Moreover, the characteristic polynomial
 \begin{equation}\label{eq-3CP}
     \tau^3-\sum_{j=0}^{2} A_{(3-j)}(t,\xi)\tau^j=\prod_{j=1}^3(\tau-\lambda_j(t,\xi)),
 \end{equation}
has the real-valued roots $\lambda_1, \lambda_2, \lambda_3 \in C^1([0,T],S^1(\mathbb R^{2n}))$. 
 
By setting $U=(u_1, u_2, u_3)^T$ with $u_k=D_t^{k-1}\lara{D_x}^{3-k}u$ for $k=1,\cdots,3$ the equation above is reduced to  \begin{equation}\label{eq-3HS} 
    \begin{cases}
        D_t U = A(t,D_x) U + B(t,x,D_x)U + F(t,x), & (t,x)\in [0,T]\times \mathbb R^n,\\
        U|_{t=0} = U_0, & x \in \mathbb R^n,
    \end{cases}
\end{equation}
where $U_0 = (u_1(0), u_2(0), u_3(0))^T$, $F(t,x)= (0, 0, f(t,x))^T$, with the operators $A(t,D_x)$ and $B(t,x,D_x)$ given by 
\begin{equation*}
    A(t,D_x) = \begin{pmatrix}
        0 & \lara{D_x} & 0 \\
        0 & 0 & \lara{D_x} \\
        b_{(1)} & b_{(2)} & b_{(3)}
    \end{pmatrix}
\end{equation*}
and
\begin{equation*}
    B(t,x,D_x) =  \begin{pmatrix}
        0 & 0 & 0\\
        0 & 0 & 0\\
        b_1 - b_{(1)} & b_2- b_{(2)}  & b_3 - b_{(3)}
    \end{pmatrix}. 
\end{equation*}
The entries of $A(t,D_x)$ and $B(t,x,D_x)$ are given by  
\begin{align*}
    b_1 = A_3(t,x,D_x) \lara{D_x}^{-2} \quad & \text{ and } \quad  b_{(1)} = A_{(3)}(t,D_x) \lara{D_x}^{-2},\\
    b_2 = A_2(t,x,D_x)\lara{D_x}^{-1}\quad & \text{ and } \quad b_{(2)} = A_{(2)}(t,D_x)\lara{D_x}^{-1}, \\
    b_3 = A_1(t,x,D_x)\lara{D_x}^{0}\quad & \text{ and } \quad b_{(3)} = A_{(1)}(t,D_x)\lara{D_x}^{0}.
\end{align*}
By implementing the transformation $U=TV$ on our system with  
\begin{equation*}
    T = \begin{pmatrix}
        1 & 0 & 0 \\
        \lambda_1\lara{\xi}^{-1} & 1 & 0 \\
        \lambda_1^2\lara{\xi}^{-2} & (\lambda_1 + \lambda_2)\lara{\xi}^{-1} & 1
    \end{pmatrix} \,\,\, \text{ and } \,\, T^{-1} = \begin{pmatrix}
        1 & 0 & 0\\
     -\lambda_1\lara{\xi}^{-1} & 1 & 0\\
  \lambda_1\lambda_2\lara{\xi}^{-2}& -(\lambda_1+\lambda_2)\lara{\xi}^{-1} & 1
    \end{pmatrix},
\end{equation*}
we are led to
 \[
 D_tV=T^{-1}ATV+T^{-1}BTV-T^{-1}(D_tT)V+F.
 \]
Note that $(T^{-1} A T)(t,D_x)=P(t,D_x)$ where 
 \begin{align*}
    P = T^{-1} A T = \begin{pmatrix}
        \lambda_1(t,\xi) & \lara{\xi} & 0 \\
        0 & \lambda_2(t,\xi) & \lara{\xi} \\
        0& 0&  \lambda_3(t,\xi)
    \end{pmatrix},
\end{align*}
and 
\[
(T^{-1}BT)(t, x,D_x)-T^{-1}(D_tT)(t,D_x)=L(t,x,D_x)=D(t,x,D_x)-E(t,D_x),
\]
where
 \begin{align*}
  D(t,x,\xi) =   T^{-1}BT(t,x,\xi)
    = \begin{pmatrix}
        0 & 0 & 0\\
         0 & 0 & 0\\
        d_{3,1} & d_{3,2} & d_{3,3},
    \end{pmatrix},
 \end{align*}
 with 
 \[
 \begin{split}
 d_{3,1} &=  b_1-b_{(1)}+(b_2-b_{(2)})\lambda_1\lara{\xi}^{-1}+(b_3-b_{(3)})\lambda_1^2\lara{\xi}^{-2},\\
 d_{3,2} &= b_2-b_{(2)}+(b_3-b_{(3)})(\lambda_1+\lambda_2)\lara{\xi}^{-1},\\
 d_{3,3}&=b_3-b_{(3)},
 \end{split}
 \] 
 and
 \[
 \begin{split}
 E(t,\xi) &=    (T^{-1}D_tT)(t,\xi)\\
 & =
   \begin{pmatrix}
        0 & 0 & 0\\
     D_t\lambda_1\lara{\xi}^{-1} & 0 & 0\\
    -D_t\lambda_1\lara{\xi}^{-2}(\lambda_1+\lambda_2)+2\lambda_1D_t\lambda_1\lara{\xi}^{-2} & D_t(\lambda_1+\lambda_2)\lara{\xi}^{-1} & 0
    \end{pmatrix}.
 \end{split}
 \]
The Levi conditions \eqref{LC-1} formulated in Theorem \ref{main_theo} are given by
\[
\begin{split}
e_{2,1} &\in C([0,T],S^{-1}(\mathbb{R}^{2n})), \\
d_{3,2} - e_{3,2}&\in C([0,T],S^{-1}(\mathbb{R}^{2n})),\\
d_{3,1} - e_{3,1}&\in C([0,T],S^{-2}(\mathbb{R}^{2n})).	
\end{split}
\]
This means
\[
\begin{split}
e_{2,1}&=D_t\lambda_1\lara{\xi}^{-1} \in C([0,T],S^{-1}(\mathbb{R}^{2n})),\\
d_{3,2}-e_{3,2}&=  b_2-b_{(2)}+(b_3-b_{(3)})(\lambda_1+\lambda_2)\lara{\xi}^{-1}-D_t(\lambda_1+\lambda_2)\lara{\xi}^{-1} \in C([0,T],S^{-1}(\mathbb{R}^{2n})),\\
d_{3,1}-e_{3,1}&=b_1-b_{(1)}+(b_2-b_{(2)})\lambda_1\lara{\xi}^{-1}+(b_3-b_{(3)})\lambda_1^2\lara{\xi}^{-2}+D_t\lambda_1\lara{\xi}^{-2}(\lambda_1+\lambda_2)\\
&-2\lambda_1D_t\lambda_1\lara{\xi}^{-2}\in C([0,T],S^{-2}(\mathbb{R}^{2n})).
\end{split}
\]
\begin{remark}
Since the roots $\lambda_i$ are symbols of order $1$, the Levi conditions above imply that $\lambda_1$ is constant. The equation can have multiplicity up to order $3$ as illustrated by the following example.  \end{remark}

\begin{example}
\label{ex_3_3}
Let us work in $\R$ and set
	\[
	\begin{split}
	\lambda_1 & = \xi,\\
	\lambda_2 & = a(t)\xi,\\
	\lambda_3 & = b(t)\xi.
	\end{split}
	\]
	The corresponding characteristic polynomial is
	\[
	(\tau-\xi)(\tau - b(t)\xi)(\tau-a(t)\xi)= \tau^3 - (a(t)+b(t)+1)\xi\tau^2 + (a(t)b(t)+a(t)+b(t))\xi^2 \tau - a(t)b(t)\xi^3.
	\]
	Hence, we are considering the operator
	\[
	D_t^3-(a(t)+b(t)+1)D_xD_t^2+ (a(t)b(t)+a(t)+b(t))D_x^2D_t+a(t)b(t)D_x^3
	\]
	with lower order terms
	\[
	a_{0,2}(t,x)D_x^2 + a_{1,1}(t,x)D_xD_t+a_{2,0}(t,x)D_t^2 + a_{0,1}(t,x)D_x + a_{1,0}(t,x)D_t+ a_{0,0}(t,x).
	\]
	Hence,
	\[
	\begin{split}
	b_1-b_{(1)} & =(a_{0,2}(t,x)\xi^2 + a_{0,1}(t,x)\xi + a_{0,0}(t,x))\lara{\xi}^{-2},\\
	b_2 - b_{(2)} & =(a_{1,1}(t,x)\xi + a_{1,0}(t),x)\lara{\xi}^{-1},\\
	b_3-b_{(3)} & = a_{2,0}(t,x).
	\end{split}
	\]
	The first Levi condition 
	\[
d_{32} - D_t\lambda_2\lara{\xi}^{-1}\in C([0,T], S^{-1}(\R^{2n})),
	\]
	where $	d_{32}= b_2-b_{(2)}+(b_3-b_{(3)})(\lambda_1+\lambda_2)\lara{\xi}^{-1}$. Then we have
	\begin{align*}
	d_{32} - D_t\lambda_2\lara{\xi}^{-1}	&= b_2-b_{(2)}+(b_3-b_{(3)})(\lambda_1(t,\xi)+\lambda_2(t,\xi))\lara{\xi}^{-1} - D_t\lambda_2\lara{\xi}^{-1} \\
		&= (a_{1,1}(t,x)\xi + a_{1,0}(t))\lara{\xi}^{-1} + a_{2,0}(t,x)(\xi +a(t)\xi)\lara{\xi}^{-1} -  D_ta(t)\xi\lara{\xi}^{-1} \\
		&=( (a_{1,1}(t,x) + (1+a(t)a_{2,0}(t,x) - D_ta(t))\xi + a_{1,0}(t,x))\lara{\xi}^{-1} \in C([0,T], S^{-1}(\R^{2n})).
	\end{align*}
This can be obtained by setting 
	\begin{equation*}
		a_{1,1}(t,x) + (1+a(t))a_{2,0}(t,x) - D_ta(t) = 0.
	\end{equation*}
The second Levi condition is written as 
		\[
	\begin{split}
	d_{31} &=b_1-b_{(1)}+(b_2-b_{(2)})\lambda_1\lara{\xi}^{-1}+(b_3-b_{(3)})\lambda_1^2\lara{\xi}^{-2}\\
	&=b_1 - b_{(1)} + (a_{1,1}(t)\xi + a_{1,0}(t))\xi\lara{\xi}^{-2} + a_{2,0}(t)\xi^2\lara{\xi}^{-2}\\
	&\in  C([0,T], S^{-2}(\R^{2n})), 
	\end{split}
	\]
	i.e., 
	\[
	\begin{split}
	&b_1- b_{(1)}+(b_2-b_{(2)})\xi\lara{\xi}^{-1}+(b_3-b_{(3)})\xi^2\lara{\xi}^{-2}\\
	&=(a_{0,2}(t,x)\xi^2 + a_{0,1}(t,x)\xi + a_{0,0}(t,x))\lara{\xi}^{-2}+(a_{1,1}(t,x)\xi + a_{1,0}(t,x))\xi\lara{\xi}^{-2} + a_{2,0}(t,x)\xi^2\lara{\xi}^{-2}\\
	&=(a_{0,2}(t,x) + a_{1,1}(t,x) + a_{2,0}(t,x))\xi^2\lara{\xi}^{-2} + (a_{0,1}(t,x) + a_{1,0}(t,x))\xi\lara{\xi}^{-2} + a_{0,0}(t,x)\lara{\xi}^{-2}\\
	&\in C([0,T], S^{-2}(\R^{2n})).
	\end{split}
	\]
	This can be easily obtained by setting
	\[
	\begin{split}
	a_{0,2}(t,x) + a_{1,1}(t,x) + a_{2,0}(t),x&=0,\\
	a_{0,1}(t,x) + a_{1,0}(t,x)&=0.
	\end{split}
	\]
	Summarising, we choose lower order terms fulfilling
	\[
	\begin{split}
	a_{1,1}(t,x) + (1+a(t))a_{2,0}(t,x) - D_ta(t) &= 0,\\
	a_{0,2}(t,x) + a_{1,1}(t,x) + a_{2,0}(t,x)&=0,\\
	a_{0,1}(t,x) + a_{1,0}(t,x)&=0.
	\end{split}
	\]
Note that in this example we have multiplicity $1$ when $1\neq a(t)\neq b(t)$, multiplicity $2$ when $a(t)=1\neq b(t)$, $a(t)\neq 1=b(t)$, $a(t)=b(t)\neq 1$ and multiplicity $3$ when $a(t)=b(t)=1$.	
\end{example}
 
\section{Comparison with previous $C^\infty$ well-posedness results}
In this section we compare our result with the few results of $C^\infty$ well-posedness known for higher order hyperbolic equations with multiplicities and we provide some explanatory examples. We begin by saying that a general treatment of $m$-order hyperbolic equations with time and space dependent coefficients is still missing. Note that the recent result obtained in \cite{SpaTa-22} is valid only for equations with space dependent coefficients in space dimension 1. The focus of this paper is to allow variable multiplicities of any order $m$, to work in any space dimension and to allow $(t,x)$-dependent lower order terms. Let us start by analysing second order hyperbolic equations.

\subsection{$C^\infty$ well-posdeness of second order hyperbolic equations}
In her seminal paper \cite{Ol-70} Oleinik consider second order hyperbolic operators of the type
\begin{multline*}
Lu=u_{tt}-\sum_{i,j=1}^n (a_{ij}(t,x)u_{x_j})_{x_i}+\sum_{i=1}^n[(b_i(t,x)u_{x_i})_t+(b_i(t,x)u_t)_{x_i}]\\
+c(t,x)u_t+\sum_{i=1}^n d_i(t,x)u_{x_i}+e(t,x)u
\end{multline*}
with smooth and bounded coefficients, i.e., coefficients in $B^\infty([0,T]\times\R^n)$, the space of smooth functions with bounded derivatives of any order $k\ge 0$. She proves that the corresponding Cauchy problem is $C^\infty$ well-posed provided that the lower order terms fulfil a specific Levi condition known nowadays as Oleinik's condition: there exist $A,C>0$ such that 
\[
t\biggl[\sum_{i=1}^n d_i(t,x)\xi_i\biggr]^2\le C\biggl\{A\sum_{i,j=1}^na_{ij}(t,x)\xi_i\xi_j-\sum_{i,j=1}^n\partial_ta_{ij}(t,x)\xi_i\xi_j\biggr\},
\]
for all $t\in[0,T]$ and $x,\xi\in\R^n$. Note that Oleinik's condition is automatically fulfilled when the coefficients of the principal part are independent of $t$ and the $d_i$'s vanish identically.
Let us consider the equation
\[
D_t^2u-a^2(t)D_x^2+ia_{0,1}(t,x)D_xu + ia_{1,0}(t,x)D_tu+ a_{0,0}(t,x)u=f(t,x)
\] 
where $a\ge 0$ is of class $C^1$ with respect to $t$. This equation has roots $\lambda_{1}(t,\xi)=-a(t)\xi$ and $\lambda_{2}(t,\xi)=a(t)\xi$ which coincide when $a$ vanishes. We assume for simplicity, that all the equation coefficients are real valued. Making use of Oleinik's formulation we get equivalently the equation $Lu=-f$ where 
\[
Lu=u_{tt}-a^2(t)u_{xx}-a_{0,1}u_x-a_{1,0}u_t-a_{0,0}(t,x)u.
\]
Hence, according to Oleinik the Cauchy problem for the equation above is $C^\infty$ well-posed when 
\beq
\label{OC}
ta_{0,1}^2(t,x)\le C(A a^2(t)-2a(t)\partial_t a(t)),
\eeq
for all $x\in\R$ and $t\in[0,T]$. 

Let us now express our Levi conditions for the equation
\[
D_t^2u-a^2(t)D_x^2+ia_{0,1}(t,x)D_xu + ia_{1,0}(t,x)D_tu+ a_{0,0}(t,x)u=f(t,x).
\] 
We have to request 
\[
d_{2,1} - e_{2,1}\in C([0,T],S^{-1}(\mathbb{R}^2)).
\]
Since
\begin{align*}
     T^{-1}BT 
    &= \begin{pmatrix}
        0 & 0 \\
        b_1 - b_{(1)} +(b_2 -b_{(2)})\lambda_1\lara{\xi}^{-1} & b_2 - b_{(2)}
    \end{pmatrix}, \\
      T^{-1}D_tT& = \begin{pmatrix}
        1 & 0 \\
        -\lambda_1\lara{\xi}^{-1} & 1
    \end{pmatrix} \begin{pmatrix}
        0 & 0 \\
        D_t\lambda_1\lara{\xi}^{-1} & 0
    \end{pmatrix} = \begin{pmatrix}
        0 & 0 \\
        D_t\lambda_1\lara{\xi}^{-1} & 0
    \end{pmatrix}.
 \end{align*}
it follows that 
\begin{align*}
    & b_1 -b_{(1)} + (b_2 - b_{(2)})\lambda_1 \lara{\xi}^{-1} - D_t\lambda_1 \lara{\xi}^{-1} \in C([0,T], S^{-1}(\mathbb{R}^{2})).
 \end{align*}
 where 
 \begin{align*}
     b_1 -b_{(1)} &:= (ia_{0,1}(t,x)\xi + a_{0,0}(t,x))\langle\xi \rangle^{-1}, \\
     b_2 -b_{(2)} &:=i a_{1,0}(t,x).
 \end{align*}
This can be equivalently formulated as 
\[
(ia_{0,1}(t,x)+ia_{1,0}(t,x)a(t)-D_ta(t))\xi\lara{\xi}^{-1}\in C([0,T], S^{-1}(\mathbb{R}^{2})).
\]
Our Levi condition is fulfilled by choosing $a_{0,1}$ and $a_{1,0}$ such that 
\beq
\label{OurC}
a_{0,1}(t,x)+a_{1,0}(t,x)a(t)-\partial_ta(t)=0.
\eeq
Note that \eqref{OurC} does not imply in general \eqref{OC}. Indeed, assuming that $a_{1,0}(t,x)=0$ then we end up with 
\[
a_{0,1}(t,x)=\partial_ta(t).
\]
In order for \eqref{OC} to hold we need 
\[
\begin{split}
t(\partial_t a(t))^2&\le C(Aa^2(t)-2a(t)\partial_ta(t)),\\
t(\partial_t a(t))^2&\le Ca(t)(Aa(t)-2\partial_t a(t)).
\end{split}
\]
This is not true for $a(t)=t$. Indeed, it would lead to 
\[
t\le Ct(At-2)=CAt^2-2Ct\quad \Leftrightarrow \quad (1+2C)t\le CAt^2 \quad \Leftrightarrow \quad  \frac{1}{t}\le \frac{CA}{1+2C},
\]
which does not hold uniformly on $(0,T]$. This example shows that our Levi conditions allow us to handle second order hyperbolic equations that do not fulfil Oleinik's condition.

\begin{example}
\label{ex-order-2}
As an explanatory example let us consider the following characteristic roots: 
\begin{align*}
\lambda_1 = t^{2\alpha} \xi, \quad \text{ and } \quad \lambda_2 = -t^{2\beta}\xi.
\end{align*} 
Hence, we consider the following operator 
\begin{equation*}
D_t^2 + (t^{2\alpha} - t^{2\beta}) D_xD_t - t^{2\alpha + 2\beta}D_x^2,
\end{equation*}
with, for the sake of simplicity, $t$-dependent lower order terms 
\begin{equation*}
a_{0,1}(t) D_x + a_{0,0}(t) + a_{1,0}(t)D_t.
\end{equation*}
The Levi condition $d_{2,1} - e_{2,1}\in C([0,T],S^{-1}(\mathbb{R}^2))$ can be written as
\[
b_1 -b_{(1)} + (b_2 - b_{(2)})\lambda_1 \lara{\xi}^{-1} - D_t\lambda_1 \lara{\xi}^{-1} \in C([0,T], S^{-1}(\mathbb{R}^{2})).
\] 
and holds if 
\begin{align*}
a_{0,1}(t) + t^{\alpha} a_{1,0}(t) &= D_t  t^{\alpha}.
\end{align*}
In the sequel we give few examples of coefficients fulfilling the Levi condition above and the corresponding second order operator. Note that no assumptions are required on $a_{0,0}(t)$ a part from the standard continuity.
\begin{itemize}
	\item Let $a_{0,1} = 0$ and $a_{1,0} = i^{-1} \alpha t^{-1}$. This leads to
	\begin{equation*}
	D_t^2 u + (t^{2\alpha} - t^{2\beta})D_tD_x u - t^{2\alpha}D_x^2 u - i^{-1} \alpha t^{-1}D_tu+a_{0,0}(t)u =0.
	\end{equation*} 
	\item Let $a_{1,0} = 0$ and $a_{0,1} = i^{-1} \alpha t^{\alpha -1}$. In this case, we get
	\begin{equation*}
	D_t^2 u + (t^{2\alpha} - t^{2\beta})D_tD_x u  - t^{2\alpha}D_x^2 u - i^{-1} \alpha t^{\alpha -1}D_x u+a_{0,0}(t)u  =0.
	\end{equation*}
	\item Let $a_{0,1} = i^{-1} \frac{\alpha}{2} t^{\alpha -1}$ and $a_{1,0} = i^{-1} \frac{\alpha}{2} t^{ -1}$. This case leads to
	\begin{equation*}
	D_t^2 u + (t^{2\alpha} - t^{2\beta})D_tD_x u  - t^{2\alpha}D_x^2 u - i^{-1} \frac{\alpha}{2} t^{\alpha -1}D_x u - i^{-1} \frac{\alpha}{2} t^{-1}D_t u+a_{0,0}(t)u =0.
	\end{equation*}
\end{itemize}

\end{example}
\begin{example}
As another example let us consider the characteristic roots: 
\begin{align*}
\lambda_1 = \cos^2(t) \xi, \quad \text{ and } \quad \lambda_2 =1+\sin^2(t)\xi,
\end{align*} 
where the double multiplicity is represented by

\includegraphics[scale=0.6]{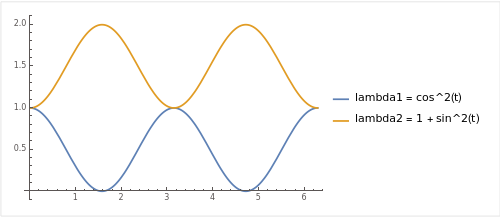}

Hence, we consider the following operator 
\begin{equation*}
D_t^2 - 2 D_xD_t + \cos^2(t)(1+\sin^2(t))D_x^2,
\end{equation*}
with lower order terms 
\begin{equation*}
a_{0,1}(t,x) D_x + a_{0,0}(t,x) + a_{1,0}(t,x)D_t
\end{equation*}
will fulfil the Levi condition 
\begin{align*}
a_{0,1}(t) + \cos(t) a_{1,0}(t) &= D_t  \cos(t)
\end{align*}
if $a_{1,0}(t) = 0$, then it leads to $a_{0,1}(t) = i\sin(t)$.
\end{example}
We conclude this subsection with a higher-dimensional example. Note that many results available in the literature hold only in space dimension 1, e.g., \cite{SpaTag-07, SpaTa-22}, so our result has he technical advantage of being formulated in any space dimension.
\begin{example}
\label{ex_R_2}
The second order operator with principal part 
\[
D_t^2-a_1(t)D_{x_1}D_t-a_2(t)D_{x_2}D_t+a_1(t)a_2(t)D_{x_1}D_{x_2}
\]
has characteristic roots $\lambda_1(t,\xi)=a_1(t)\xi_1$ and $\lambda_1(t,\xi)=a_2(t)\xi_2$ which coincides when $a_1(t)=a_2(t)=0$. Writing the lower order terms in operator form as
\[
\sum_{j=1}^2 a^{(j)}_{0,1}(t)D_{x_j}+a_{1,0}(t)D_t+a_{0,0}(t,x),
\]
we can now formulate the Levi condition 
\[
b_1 -b_{(1)} + (b_2 - b_{(2)})\lambda_1 \lara{\xi}^{-1} - D_t\lambda_1 \lara{\xi}^{-1} \in C([0,T], S^{-1}(\mathbb{R}^{4})),
\] 
as
\[
\biggl(\sum_{j=1}^2 a^{(j)}_{0,1}(t)\xi_j+a_{1,0}(t)\lambda_1(t,\xi)-D_t\lambda_1(t,\xi)\biggr)\lara{\xi}^{-1}\in C([0,T], S^{-1}(\mathbb{R}^{4})).
\]
This leads to 
\[
\biggl(\sum_{j=1}^2 a^{(j)}_{0,1}(t)\xi_j+a_{1,0}(t)a_1(t)\xi_1-D_ta_1(t)\xi_1\biggr)\lara{\xi}^{-1}\in C([0,T], S^{-1}(\mathbb{R}^{4}))
\]
which is fulfilled by choosing $a^{(2)}_{0,1}=0$ and 
\[
a^{(1)}_{0,1}(t)+a_{1,0}(t)a_1(t)=D_ta_1(t)
\]
\end{example}

\subsection{$C^\infty$ well-posdeness of third order hyperbolic equations}
Our result requires a limited level of regularity on the coefficients. Indeed, we need continuity with respect to $t$ and that the roots $\lambda_i(t,\xi)$ are of class $C^1$ with respect to $t$.
Note that in \cite{ColGrOrTa-23} the coefficients of the principal part are assumed to be of class $C^2$ and all the coefficients, also the lower order terms, need to be time-dependent only.
We can therefore provide bigger generality under this point of view. As an explanatory example, let us consider a third order operator with characteristic roots: 
	\begin{align*}
		\lambda_1 = 0, \quad \lambda_2 = a(t)\xi, \quad  \lambda_3 =b(t)\xi,
	\end{align*} 
where $a,b\in C^1([0,T])$. This is an example of variable multiplicity, since we have multiplicity $3$ when $a(t)=b(t)=0$, multiplicity $2$ when $a(t)=b(t)\neq 0$, $a(t)=0\neq b(t)$ and $a(t)\neq 0=b(t)$, and multiplicity $1$ otherwise. We can therefore consider the operator 
\[
	D_t^3-(a(t)+b(t))D_x D_t^2 + a(t)b(t) D_x^3,
	\]
	with lower order terms
	\[
	a_{0,2}(t,x)D_x^2 + a_{1,1}(t,x)D_xD_t+a_{2,0}(t,x)D_t^2 + a_{0,1}(t),xD_x + a_{1,0}(t,x)D_t+ a_{0,0}(t,x).
	\]
Note that for the sake of simplicity we work in space dimension $1$.	Our Levi conditions on the lower order terms are formulated as follows:
\[
\begin{split}
	a_{1,1}(t,x) + a(t)a_{2,0}(t,x) - D_ta(t) &= 0,\\
	a_{0,2}(t,x) + a_{1,1}(t,x) + a_{2,0}(t,x)&=0,\\
	a_{0,1}(t,x) + a_{1,0}(t,x)&=0.
	\end{split}
	\]
	
 Note that if we choose $a_{2,0}=f(t)+g(x)$ then the first Levi condition holds when 
 \[
 \begin{split}
	a_{1,1}(t,x) + a(t)g(x)&= 0,\\
	a(t)f(t)-D_ta(t)&=0.
	\end{split}
 \]
 \begin{example}
 Let us select the characteristic roots:
\begin{equation*}
	\lambda_1 = \xi, \quad \lambda_2 = \cos(t)\xi, \quad \lambda_3 = (1+\sin^2(t))\xi,
\end{equation*}
where the variable  multiplicities are represented by  

\includegraphics[scale=0.6]{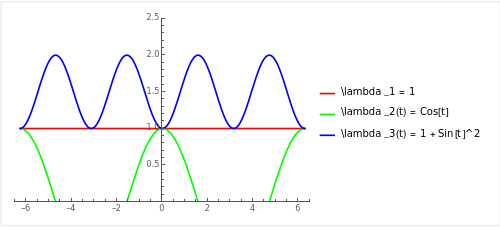}

Hence, we are considering the operator
\begin{align*}
	D_t^3 + (\cos(t)\sin^2(t)+ 2\cos(t) + 1 + \sin^2(t)  )D_x^2D_t \\
	-(\cos(t)+\sin^2(t)+2)D_x D_t^2 + \cos(t)(1+\sin^2(t)) D_x^3,
\end{align*}
with lower order terms
\[
a_{0,2}(t,x)D_x^2 + a_{1,1}(t,x)D_xD_t+a_{2,0}(t,x)D_t^2 + a_{0,1}(t,x)D_x + a_{1,0}(t,x)D_t+ a_{0,0}(t,x).
\]
Then Levi conditions on the lower order terms: 
\begin{align*}
a_{1,1}(t,x) + (1+\cos(t))a_{2,0}(t,x) &= i\sin(t), \\
a_{0,2}(t,x) + a_{1,1}(t,x) + a_{2,0}(t,x)&=0,\\
a_{0,1}(t,x) + a_{1,0}(t,x)&=0.
\end{align*}
Note that by choosing $a_{1,1}=0$ and an interval $[0,T]$ where $ (1+\cos(t))\neq 0$ then both $a_{2,0}$ and $a_{0,2}$ are time-dependent and can be expressed as 
\[
a_{2,0}=i\frac{\sin(t)}{1+\cos(t)}
\]
with $a_{0,2}=-a_{2,0}$.
\end{example}

\subsection{$C^\infty$ well-posedness of fourth order hyperbolic equations}
We conclude this section with an explanatory fourth order example.
\begin{example}
\label{ex_4_4}
	Let us work in $\mathbb{R}$ and set 
	\begin{equation*}
		\lambda_1 = \xi, \,\, \lambda_2 = -\xi, \,\,\, \lambda_3 = \sqrt{a(t)}\xi, \,\,\, \lambda_4 =-\sqrt{a(t)}\xi,
 	\end{equation*}
 	where $a(t)\ge 0$. The corresponding characteristic polynomial is 
 	\begin{equation*}
 		(\tau^2 - \xi^2) (\tau^2 - a(t)\xi^2) = \tau^4 - (a(t) + 1) \xi^2 \tau^2 + a(t) \xi^4.
  	\end{equation*}
  	Hence, we consider the operator 
  	\begin{equation*}
  		D_t^4 - (a(t) +1)D_x^2 D_t^2 + a(t)D_x^4
  	\end{equation*}
  	with lower order terms
  	\begin{align*}
  		b_1 - b_{(1)} &= (a_{0,3}(t,x)\xi^3 + a_{0,2}(t,x)\xi^2 + a_{0,1}(t,x)\xi + a_{0,0} (t,x))\lara{\xi}^{-3}, \\
  		b_2 - b_{(2)} &= (a_{1,3}(t,x)\xi^3 + a_{1,2}(t,x)\xi^2 + a_{1,1}(t,x)\xi + a_{1,0} (t,x))\lara{\xi}^{-2},\\
  		b_3 - b_{(3)} &= (a_{2,1}(t,x)\xi + a_{2,0}(t,x)  )\lara{\xi}^{-1},\\
  		b_4 - b_{(4)} &= a_{3,1}(t,x) \xi + a_{3,0}(t,x).
  	\end{align*}
  	The first Levi condition 
  	\begin{align*}
  		d_{43} -e_{43} &= b_3-b_{(3)} + (b_4 - b_{(4)})\lambda_3\lara{\xi}^{-1} - D_t\lambda_3\lara{\xi}^{-1}\\
  		&= (a_{2,1}\xi + a_{2,0}  +  (a_{3,1}\xi + a_{3,0})\sqrt{a(t)} \xi - D_t\sqrt{a(t)}\xi) \lara{\xi}^{-1},
  	\end{align*}
  	is fulfilled if 
  	\begin{align*}
	& a_{2,1}(t,x) + \sqrt{a(t)} a_{3,0}(t,x) =D_t\sqrt{a(t)},\\  	
	& a_{3,1}(t,x) = 0.
  	\end{align*} 
  The second Levi condition 
  \begin{align*}
  	d_{42} - e_{42} & = b_2-b_{(2)}  + (b_4 - b_{(4)}) (\lambda_1^2 + \lambda_2^2 + \lambda_1\lambda_2)\lara{\xi}^{-2} \\
  	& = (a_{1,3}\xi^3 + a_{1,2}\xi^2 + a_{1,1}\xi + a_{1,0} + (a_{3,1} \xi + a_{3,0}) \xi^2) \lara{\xi}^{-2}
  \end{align*}
  	is fulfilled if 
  	\begin{align*}
  		& a_{1,3} (t,x) + a_{3,1}(t,x)  = 0,\\
  		& a_{1,2}(t,x) + a_{3,0} (t,x) = 0,\\
  		& a_{1,1} (t,x) = 0.
  	\end{align*}
  	The third Levi condition 
  	\begin{align*}
  		d_{41} - e_{41} &= b_1-b_{(1)} + (b_2 - b_{(2)})\lambda_1\lara{\xi}^{-1} + (b_3 - b_{(3)})\lambda_1^2\lara{\xi}^{-2} + (b_4-b_{(4)})\lambda_1^3\lara{\xi}^{-3} \\
  		&= (a_{0,3}\xi^3 + a_{0,2}\xi^2 + a_{0,1}\xi + a_{0,0} + a_{1,3}\xi^4 + a_{1,2}\xi^3 + a_{1,1}(t,x)\xi^2 + a_{1,0}\xi) \lara{\xi}^{-3}\\ 
  		& + (a_{2,1}\xi^3 + a_{2,0} \xi^2  + a_{3,1} \xi^4 + a_{3,0}\xi^3) \lara{\xi}^{-3}
  	\end{align*}
  	should be fulfilled if 
  	\begin{align*}
  		& a_{0,3} (t,x)+ a_{1,2} (t,x)+ a_{2,1} (t,x)+ a_{3,0}(t,x) = 0,\\
  		& a_{0,2} (t,x)+ a_{1,1}(t,x) + a_{2,0} (t,x)= 0, \\
  		& a_{0,1} (t,x)+ a_{1,0} (t,x)= 0, \\
  		& a_{3,1} (t,x)+ a_{1,3}(t,x)= 0.
  	\end{align*}
  	In conclusion, we have identified the following Levi conditions on the lower order terms:
  	\begin{align*}
  	&	a_{2,1}(t,x) + \sqrt{a(t)}a_{3,0}(t,x) = D_t\sqrt{a(t)} , \\
  	&   a_{1,2}(t,x) + a_{3,0}(t,x) = 0, \\
  	&  	a_{2,1}(t,x) + a_{0,3}(t,x) = 0, \\
  	&	a_{0,2}(t,x) + a_{2,0}(t,x) = 0, \\
  	&	a_{0,1}(t,x) + a_{1,0}(t,x) = 0,\\
  	& 	a_{1,3} = a_{3,1} = a_{1,1} = 0.
  	\end{align*}
\end{example}

%
%


\begin{thebibliography}{99}
\bibitem{Ber-93} 
E. Bernardi:
The Cauchy problem for hyperbolic operators with triple involutive characteristics. Publ. Res. Inst. Math. Sci. 29(1), 23-28 (1993)

\bibitem{BerBovPet-15}
E. Bernardi, A. Bove, V. Petkov: Cauchy problem for effectively hyperbolic operators with triple characteristics of variable multiplicity. J. Hyperbolic Differ. Equ. 12, 535-579 (2015)

\bibitem{BerNis-15}
E. Bernardi, T. Nishitani: Counterexamples to $C^{\infty}$ well posedness for some hyperbolic operators with triple characteristics. Proc. Japan Acad. Ser. A Math. Sci. 91, 19-24 (2015)

\bibitem{BerPP:12}
E. Bernardi, P. Parenti and P. Parmeggiani:
The Cauchy problem for hyperbolic operators with double characteristics in presence of transition. Comm. Partial Differential Equations. 37(7), 1315-1356 (2012)

\bibitem{Cha-74} 
J. Chazarain: Op\'erateurs hyperboliques a caract\'eristiques de multiplicit\'e constante. Ann. Inst. Fourier (Grenoble) 24(1), 173-202 (1974)

\bibitem{ColDeSpa-79}
F. Colombini, E. De Giorgi, S. Spagnolo: Sur les \'equations hyperboliques avec des coefficients qui ne d\'ependent que du temps. Ann. Scuola Norm. Sup. Pisa Cl. Sci. (4) 6(3), 511-559 (1979)

\bibitem{ColGrOrTa-23} 
F. Colombini, T. Gramchev, N. Orr\'u, G. Taglialatela: On some hyperbolic equations of third order. Annali di Matematica Pure ed Applicata. 202, 143-183 (2023)

\bibitem{ColJanSpa-83}
F. Colombini, E. Jannelli, S. Spagnolo: Well-posedness in the Gevrey classes of the Cauchy problem for a nonstrictly hyperbolic equation with coefficients depending on time. Ann. Scuola Norm. Sup. Pisa Cl. Sci. (4) 10(2), 291-312 (1983)

\bibitem{ColOr-99}
F. Colombini, N. Orr\'u: Well-posedness in $C^{\infty}$ for some weakly hyperbolic equations. J. Math. Kyoto Univ. 39(3), 399-420 (1999)

\bibitem{ColSpa-82}
F. Colombini, S. Spagnolo: An example of a weakly hyperbolic Cauchy problem not well posed in $C^{\infty}$. Acta Math. 148, 243-253 (1982)

\bibitem{ColTag-06}
F. Colombini, G. Taglialatela: Well-posedness for hyperbolic higher order operators with finite degeneracy. Kyoto J. Math. 46(4), 833-877 (2006)

\bibitem{ColGraOrruTag-23}
F. Colombini, T. Gramchev, N. Orr\'u, G. Taglialatela: On some hyperbolic equations of third-order. Annali di Mathematica Pura ed Applicata 202, 143-183 (2023)

\bibitem{DeP-72}
J.C. De Paris: Probl\'eme de Cauchy oscillatoire pour un op\'erateur diff\'erentiel \'a caract\'eristiques multiples; lien avec l'hyperbolicit\'e. J. Math. Pures et Appl. 51, 231-256 (1972) 

\bibitem{Dunn-75}
J.L. Dunn: A sufficient condition for hyperbolicity of partial differential operators with constant coefficient principal part. Trans. Am. Math. Soc. 201, 315-327 (1975)

\bibitem{FlaStr-71}
H. Flaschka, G. Strang: The correctness of the Cauchy problem. Adv. Math. 6(3), 347-379 (1971)

\bibitem{Gar-51}
L. G\r{a}rding: Linear hyperbolic partial differential equations with constant coefficients. Acta Math. 85, 1-62 (1951)



\bibitem{GarJRuz}
C. Garetto, C. J\"ah and M. Ruzhansky: Hyperbolic systems with non-diagonalisable principal part and variable multiplicities, {I}: well-posedness.
Math. Ann. 372(3-4), 1597-1629 (2018)

\bibitem{GarJRuz2} 
C. Garetto, C. J\"ah and M. Ruzhansky:  Hyperbolic systems with non-diagonalisable principal part and variable multiplicities, {I}: microlocal analysis. J. Diff. Eq., 269(10), 7881-7905 (2020)


\bibitem{GramRuz2013}
T. Gramchev and M. Ruzhansky: Cauchy problem for $2 \times 2$ hyperbolic systems of pseudo-differential equations with nondiagonalisable principal part.
Studies in Phase Space Analysis with Applications to PDEs, Progr. Nonlinear Differential Equations Appl., 84, 129-144 (2013)

\bibitem{Hor-85}
L. H\"ormander: The Analysis of Linear Partial Differential Operators. Springer-Verlag, Berlin (1985)

\bibitem{Nish-80}
T. Nishitani: The Cauchy problem for weakly hyperbolic equations of second order. Commun. Partial Differ. Equ. 5(12), 1273-1296 (1980)

\bibitem{Nish-22}
T. Nishitani: Diagonal symmetrizers for hyperbolic operators with triple characteristics. Math. Ann. 383, 529-569 (2022)

\bibitem{Ol-70}
O.A. Oleinik: On the Cauchy problem for weakly hyperbolic equations. Comm. Pure Appl. Math. 23, 569-586, (1970)

\bibitem{PP:04}
C. Parenti and A. Parmeggiani: The Cauchy problem for certain systems with double characteristics. Osaka J. Math. 41(3), 659-680 (2004)

\bibitem{PP:09}
C. Parenti and A. Parmeggiani: On the Cauchy problem for hyperbolic operators with double characteristics. Comm. Partial Differential Equations. 34:7-9, 837-888 (2009)


\bibitem{SpaTag-07}
S. Spagnolo, G. Taglialatela: Homogeneous hyperbolic equations with coefficients depending on one space variable. J. Hyperbolic Differ. Equ. 4(3), 533-553 (2007)

\bibitem{SpaTa-22}
S. Spagnolo and T. Taglialatela: The Cauchy problem for properly hyperbolic equations in one space variable. J. Hyperbolic Differ. Equ. 19(3), 439-466 (2022)

\bibitem{Sve-70}
S.L. Svensson: Necessary and sufficient conditions for the hyperbolicity of polynomials with hyperbolic principal part. Ark. Mat. 8, 145-162 (1970)

\bibitem{Wak-80}
S. Wakabayashi: The Cauchy problem for operators with constant coefficient hyperbolic principal part and propagation of singularities. Japan. J. Math. 6, 179-228 (1980)

\bibitem{Wak-15}
S. Wakabayashi: On the Cauchy problem for a class of hyperbolic operators whose coefficients depend only on the time variable. Tsukuba J. Math. 39(1), 121-163 (2015)
 \end{thebibliography}
\end{document}